\documentclass[a4paper,12pt]{amsart}
\usepackage{amsthm,amssymb,amsmath}
\usepackage[T1]{fontenc}
\usepackage[utf8]{inputenc}
\usepackage[hidelinks]{hyperref}
\usepackage{booktabs}
\newtheorem{lemma}{Lemma}
\usepackage[left=2cm,right=2cm,top=2cm,bottom=2cm]{geometry}
\title{A nonexistence criterion and new constructions for Butson Hadamard matrices}
\author{Domonkos Czifra, M\'at\'e Matolcsi, Ferenc Sz\"oll\H{o}si}
\thanks{November 14, 2025. Preprint.\\
M.~M.~was supported by grants NKFIH-146387, NKFIH-154121 and KKP 133819. 
D.~C. was supported by the Ministry of Innovation and Technology NRDI Office within the framework of the Artificial Intelligence National Laboratory (RRF-2.3.1-21-2022-00004). F. Sz. was supported in part by JSPS KAKENHI Grant Number 24K06829}
\usepackage{booktabs}
\newtheorem{theorem}{Theorem}
\newtheorem{conjecture}{Conjecture}
\theoremstyle{definition}
\newtheorem{example}{Example}

\newcommand{\Z}{\mathbb{Z}}
\newcommand{\R}{\mathbb{R}}
\begin{document}
\begin{abstract}
Based on the concept of positive definite functions on finite groups, we present a new necessary condition  for the existence of Butson Hadamard matrices $BH(n,q)$. We use this condition to prove some nonexistence results for a sequence of values of $(n,q)$, and also to facilitate a computer search and discover a matrix $BH(18,14)$.

Furthermore, we use cyclotomic cosets to construct matrices $BH(34,10)$, $BH(62,6)$, $BH(82,6)$, and $BH(146,6)$ for the first time. These matrices have  a $2$-circulant structure.
\end{abstract}

\maketitle

\section{Introduction}

Butson Hadamard matrices $BH(n,q)$ are $n\times n$ complex matrices with all entries being $q$th roots of unity, and the rows (or, equivalently, the columns) being orthogonal to each other with respect to the standard complex inner product $\left\langle\cdot,\cdot\right\rangle$ of $\mathbb{C}^n$. This class of matrices provides a natural generalization of \textit{real} Hadamard matrices. The fundamental question of characterizing the pairs $(n,q)$ for which $BH(n,q)$ matrices exist remains wide open. At present, our knowledge is restricted to some sporadic constructions \cite{but}, \cite{daw}, \cite{MW}, \cite{sebyam} and a few nonexistence criteria in the literature \cite{BBS}, \cite{BROCK}, \cite{dL}, \cite{LL}, \cite{W}. 

\medskip

The theory of positive definite functions on finite groups was recently applied in \cite{M} to study real Hadamard matrices. In Section~\ref{sec2} we revisit this approach and adapt it to the case of Butson Hadamard matrices $BH(n,q)$. As a result, we present a new necessary condition for their existence. On the one hand, this condition can naturally be used to prove nonexistence results for certain values $(n,q)$. On the other hand, for other values $(n,q)$, it implies some necessary properties that a $BH(n,q)$ matrix must satisfy and, as such, it facilitates a computer search for such matrices. We have carried out such a search for $BH(18,14)$, and present such a matrix in Section \ref{sec3}. The existence of this matrix fills a gap in the online catalogue of Butson Hadamard matrices of relatively small order\footnote{See \url{https://wiki.aalto.fi/display/Butson/Matrices+up+to+monomial+equivalence}}. Based on this discovery and other smaller examples \cite[p.~105]{AA}, \cite{LOS}, we make the following conjecture. 

\begin{conjecture}\label{conj1}
For any prime $p$, Butson Hadamard matrices $BH(2p+4, 2p)$ exist.     
\end{conjecture}

Butson matrices $BH(n,q)$ for prime powers $q=p^r$ have been extensively studied. It is well-known that a $BH(n,p^r)$ matrix can only exist when $n=1$, or $n$ is a multiple of $p$, see \cite{W}. For $q=4$ it is conjectured that this divisibility condition is the only restriction \cite{sebyam}. 

In a seminal paper, Compton, Craigen, and de Launey initiated the systematic study of the simplest case when $q$ is not a prime power, and constructed a series of `unreal' $BH(n,6)$ matrices \cite{unreal}. Subsequently the third author of this manuscript visited Craigen, and they constructed $BH(p^2,6)$ matrices for every prime $p$, see \cite[Theorem~1.4.41]{phd}. This result was recently generalized by Schmidt, Wong, and Xiang \cite{schmidt2025}. In Section \ref{sec4} we recall some theory of circulant matrices \cite{BJ}, and present a construction of $BH(62,6)$, $BH(82,6)$, and $BH(146,6)$ matrices, the existence of which has also been open. These examples further support the conjecture (see \cite[Conjecture~1.4.47]{phd}) that $BH(2p,6)$ matrices always exist for any prime number $p$.

\section{A necessary condition for the existence of $BH(n,q)$}\label{sec2}

Let $\zeta_q\in\mathbb{C}$ denote a  primitive complex $q$th root of unity, e.g., $\zeta_q=\mathrm{exp}(2\pi\mathbf{i}/q)$, and let $\Omega_q:=\{1,\zeta_q,\dots,\zeta_q^{q-1}\}$. The set $\Omega_q^n$ is a finite commutative group equipped with the operation of coordinate-wise multiplication. We think of elements of $\Omega_q^n$ as row vectors of length $n$. For row vectors $u,v\in\Omega_q^n$, let us denote by $u\circ v\in\Omega_q^n$ their coordinate-wise product (also known as their \textit{Hadamard product}). Let $\overline{u}$ denote the complex conjugate of $u\in\Omega_q^n$.

\medskip

Assume that a $BH(n,q)$ matrix $H$ exists, with its rows being denoted by $h_1$, $\dots$, $h_n\in \Omega_q^n$. Define the (rescaled) \textit{frequency map of $H$}, $g_H\colon \Omega_q^n\to \mathbb{R}^+_0$, as follows:
\begin{equation}\label{gh}
g_H(x):=\frac{1}{n}\cdot\left|\{(i,j)\colon x=h_i\circ\overline{h}_j,\quad i,j\in\{1,\dots,n\}\}\right|,    
\end{equation}
where the notation $|A|$ stands for the cardinality of any finite set $A$. 

That is, the value of $g_H(x)$ is a nonnegative integer, counting the number of times the Hadamard product of a row $h_i$ of $H$ with the conjugate of another row $h_j$ is equal to the vector $x$ (note that $i=j$ is not excluded). Let $e:=[1,1,\dots,1]\in \Omega_q^n$. As $h_i\circ \overline{h}_i=e$ for every $i$, we clearly have  $g_H(e)=1$. Since the rows of $H$ are orthogonal to each other, all other vectors $x$ where $g_H(x)$ might be nonzero are inside the \textit{orthogonality set}:
\[ORT:=\left\{x\in\Omega_q^n\colon \sum_{i=1}^n x_i=0\right\}.\]
As such, the \textit{support} of $g_H$ satisfies $\mathrm{supp}(g_H):=\{x\in\Omega_q^n\colon g_H(x)\neq 0\}\subseteq ORT\cup\{e\}$. Further, since the total number of coordinate pairs $(i,j)$ is $n^2$, we have:
\[\sum_{x\in\Omega_q^n}g_H(x)=\frac{1}{n}\cdot n^2=n.\]

\medskip

Let us introduce the following  notation: for $x\in\Omega_q^n$ and $z\in \mathbb{Z}_q^n=\{0,1, \dots, q-1\}^n$ let
\[x^z:=\prod_{k=1}^n x_k^{z_k}\in\Omega_q.\]
Consider the Fourier transform of $g_H$, defined as $\widehat{g_H}\colon \mathbb{Z}_q^n\to \mathbb{R}$:
\[\widehat{g_H}(z):=\sum_{x
\in \Omega_q^n} g_H(x)x^z=\sum_{i=1}^n\sum_{j=1}^n(h_i\circ\overline{h}_j)^z.\]
Simple calculation shows that the function $g_H$ is \textit{positive definite}, i.e. $\widehat{g_H}$ is nonnegative. Indeed:
\begin{align}\label{ghhat}
\widehat{g_H}(z)&=\sum_{i=1}^n\sum_{j=1}^n(h_i\circ\overline{h}_j)^z
=\left|\sum_{i=1}^nh_i^z\right|^2\geq 0.
\end{align}
While $g_H$ depends on the concrete choice of $H$, the key insight is that we can perform several averaging steps to obtain mappings with more transparent properties. Namely, let $S_n$ denote the symmetric group of order $n$, and let us consider the $n!$ column permutations of $H$, denoted by $\sigma(H)$ where $\sigma\in S_n$. Define the following average function $g_0\colon \Omega_q^n\to\mathbb{R}$,
\[g_0(x):=\frac{1}{n!}\sum_{\sigma\in S_n}g_{\sigma(H)}(x).\]
Notice that $g_0\ge 0$ and  $\widehat{g_0}=\frac{1}{n!}\sum_\sigma \widehat{g_{\sigma(H)}}\geq 0$. Moreover, the main advantage is that the function $g_0$ is invariant up to coordinate permutations. We call two row vectors $x, y\in\Omega_q^n$ \textit{permutation-equivalent}, and denote this by the symbol $x\sim y$, if and only if there is a permutation $\tau\in S_n$ such that $\tau(x)=y$. Let us define the \textit{orbit} of a vector $x$ as ${\mathrm{orb}(x):=\{y\in\Omega_q^n\colon y\sim x\}}$, and decompose the set $ORT$ into $M$ nonempty permutation-equivalence classes: $ORT=\cup_{i=1}^M\mathrm{orb}(\omega_i)$, where $\omega_i\in \mathrm{orb}(\omega_i)$ is a representative for any $i\in\{1,\dots, M\}$. It follows that $g_0$ admits at most $M+2$ distinct values, and we may write:
\[g_0(x)=\begin{cases}
1,& \text{if } x\sim e\\
c_i,& \text{if } x\sim \omega_i, i\in\{1,\dots, M\}\\
0,& \text{otherwise}.
\end{cases}\]
Clearly, the values $c_i$ must satisfy the \textit{feasibility conditions}: $0\leq c_i\leq \frac{n-1}{|\mathrm{orb}(\omega_i)|}$, together with the following linear constraint:
\begin{equation}\label{lin}
\sum_{x\in \Omega_q^n}g_0(x)=1+\sum_{i=1}^M |\mathrm{orb}(\omega_i)|c_i=n.
\end{equation}
In general, for fixed $n$ and $q$ the representatives $\omega_i$ can be determined, and the orbit sizes can be calculated. However, the values $c_i$ still depend on the choice of $H$ apart from the special case when $M=1$. In this case, using \eqref{lin} above, we have:
\begin{equation}\label{c1}
c_1=\frac{n-1}{|\mathrm{orb}(\omega_1)|}.
\end{equation}
Now consider $\widehat{g_0}(z)$, which is nonnegative for every $z\in\mathbb{Z}_q^n$:
\begin{align}
\label{g0hat}\widehat{g_0}(z)&=1+\sum_{i=1}^Mc_i\sum_{x\sim \omega_i}x^z\ge 0.
\end{align}
It is clear that $\widehat{g_0}(z)$ is invariant up to coordinate permutations, i.e. $\widehat{g_0}(z)=\widehat{g_0}(\tau(z))$ for every $\tau\in S_n$.

\medskip

Our aim is to further reduce the number of variables $c_i$, and we would like to do this by merging certain equivalence classes of $\sim$. We can do so, if we make a further averaging step over the scaled variants of $g_0$. Namely, introduce the function $g_1\colon \Omega_q^n\to \mathbb{R}$,
\[g_1(x):=\frac{1}{q}\sum_{\zeta\in\Omega_q}g_0(\zeta x).\]
With this definition, the function $g_1$ is now invariant up to coordinate permutations, and up to multiplying the vector $x$ by any scalar $\zeta\in\Omega_q$.
It is also clear, that $g_1(e)=1/q$, and $\sum_{x\in\Omega_q^n}g_1(x)=n$. Let us calculate its Fourier transform $\widehat{g}_1\colon \mathbb{Z}_q^n\to\mathbb{R}$,
\[\widehat{g}_1(z):=\sum_{x\in\Omega_q^n}g_1(x)x^z.\] We have the following simple property.
\begin{lemma}\label{L1}
The function $\widehat{g}_1\colon \mathbb{Z}_q^n\to\mathbb{R}$ satisfies
\[\widehat{g}_1(z)=\begin{cases}
\widehat{g_0}(z),& \text{when $\sum_{k=1}^n z_k\equiv\ 0\ (\mathrm{mod}\ q)$}\\
0,& \text{otherwise}.
\end{cases}\]
In particular, $\widehat{g}_1(z)\geq 0$ on $\mathbb{Z}_q^n$.
\end{lemma}
\begin{proof}
Let us use the notation $s:=-\sum_{k=1}^n z_k\in\mathbb{Z}_q$. We have
\begin{align*}
\widehat{g}_1(z)&=\sum_{x\in\Omega_q^n}g_1(x)x^z=\sum_{x\in\Omega_q^n}\frac{1}{q}\sum_{\zeta\in\Omega_q}g_0(\zeta x)x^z=\frac{1}{q}\sum_{\zeta\in\Omega_q}\sum_{x\in\Omega_q^n}g_0(\zeta x)x^z\\
&=\frac{1}{q}\sum_{\zeta\in\Omega_q}\zeta^{s}\sum_{x\in\Omega_q^n}g_0(\zeta x)(\zeta x)^z=\frac{1}{q}\sum_{\zeta\in\Omega_q}\zeta^{s}\widehat{g_0}(z)=\frac{1}{q}\cdot \widehat{g_0}(z)\cdot \sum_{k=1}^q(\zeta_q^{s})^k,
\end{align*}
and the claim follows.
\end{proof}

\medskip

We will make one final step of averaging, exploiting the automorphisms of the underlying group $\Z_q$. For any number $1\le r\le q-1$ which is co-prime to $q$, and any vector $x=[x_1, \dots, x_n]\in \Omega_q^n$ define $x^r=[x_1^r, \dots, x_n^r]$. Next, define 
\begin{equation}\label{group}
g(x)=\frac{1}{\varphi(q)}\sum_{(r,q)=1} g_1(x^r).
\end{equation}

We call two vectors $x, y\in\Omega_q^n$ \textit{permutation-scale-group-equivalent} and denote this by the symbol $x\cong y$, if and only if there is a permutation $\tau\in S_n$, a scalar $\zeta\in\Omega_q$, and a number $r$ co-prime to $q$ such that $(\zeta\cdot\tau(x))^r=y$. 

\medskip

Let us define the \textit{orbit} of a vector $x$ as $\mathrm{Orb}(x):=\{y\in\Omega_q^n\colon y\cong x\}$, and decompose the set $ORT$ into $K$ nonempty $\cong$-equivalence classes $ORT=\cup_{i=1}^K\mathrm{Orb}(\omega_i)$, where $\omega_i\in \mathrm{Orb}(\omega_i)$ is a representative for any $i\in\{1,\dots, K\}$. By the averaging above, it follows that the function $g$ is constant on these equivalence classes. As such, $g$ admits at most $K+2$ distinct values, namely:
\begin{equation}\label{ci}
g(x)=\begin{cases}
1/q,& \text{if } x\cong e\\
c_i,& \text{if } x\cong \omega_i, i\in\{1,\dots, K\}\\
0,& \text{otherwise}.
\end{cases}
\end{equation}

All of the important equations \eqref{lin}, \eqref{c1}, \eqref{g0hat} hold analogously. Namely, the values $c_i$ must satisfy the \textit{feasibility conditions}: $0\leq c_i\leq \frac{n-1}{|\mathrm{Orb}(\omega_i)|}$, together with the following linear constraint:
\begin{equation}\label{lin2}
\sum_{x\in \Omega_q^n}g(x)=1+\sum_{i=1}^K |\mathrm{Orb}(\omega_i)|c_i=n.
\end{equation}
If we have a single equivalence class, then:
\begin{equation}\label{c12}
c_1=\frac{n-1}{|\mathrm{Orb}(\omega_1)|}.
\end{equation}
Finally, $\widehat{g}(z)$ is nonnegative by Lemma~\ref{L1} and equation \eqref{group}, 
\begin{align}\label{g0hat2}
\widehat{g}(z)&=\sum_{x\in\mathrm{supp}(g)}g(x)x^z=1+\sum_{i=1}^K\sum_{x\cong \omega_i}g(x)x^z=1+\sum_{i=1}^Kc_i\sum_{x\cong \omega_i}x^z\ge 0.
\end{align}

Two vectors $z,y\in\Z_q^n$, such that $\sum_{i=1}^n z_i \equiv\sum_{i=1}^n y_i  \equiv 0 \ (\mathrm{mod} \ q)$, will be called equivalent, $y \simeq z$, if there is a permutation $\pi$ and a number $r$ co-prime to $q$ such that $z\equiv r \cdot \pi(y) \ (\mathrm{mod} \  q)$.  It is easy to see by Lemma~\ref{L1} and equation \eqref{group} that $\widehat{g}(z)=0$ whenever $\sum_{i=1}^n z_i \not\equiv 0 \ (\mathrm{mod} \ q)$, and the value of $\widehat{g}(z)$ is constant on $\simeq$-equivalence classes.  

\medskip

We summarize the results of the discussion above in the following necessary condition for the existence of a $BH(n,q)$ matrix. 

\begin{theorem}\label{main}
If a $BH(n,q)$ exists, then there exists a function $g: \Omega_q^n  \to \R$ such that $g\ge 0$, $g$ is normalized so that for any $a\in \Omega_q$ and $e_a=[a, a, \dots, a]$ we have $g(e_a)=1/q$, 
$g$ is supported on $ORT\bigcup \left (\cup_{a\in \Omega_q} e_a \right )$, $g$ is constant on $\cong$-equivalence classes of $\Omega_q^n$, $\sum_{x\in \Omega_q^n}g(x)=n$, and $\widehat{g}\ge 0$.    
\end{theorem}

This theorem gives us a tool to detect the \textit{nonexistence} of Butson Hadamard matrices $BH(n,q)$ for certain values of $(n,q)$. 
In order to show that $BH(n,q)$ matrices cannot exist, we can employ one of the following two strategies:
\begin{itemize}
\item exhibit a single $z_0\in\mathbb{Z}_q^n$, such that $\widehat{g}(z_0)<0$ for every feasible choice of $c_i$; or
\item use linear programming to show that the feasibility conditions are incompatible. In the linear program, the variables are $c_i\ge 0$ as in \eqref{ci}, the constraints are given by $\widehat{g}(z)\ge 0$ for all $z\in\Z_q^n$, and the target is to maximize the value of $1+\sum_{i=1}^K |\mathrm{Orb}(\omega_i)|c_i$ as in \eqref{lin2}. If this maximum is strictly less than $n$, then we get a contradiction with \eqref{lin2}, and conclude that $BH(n,q)$ matrices cannot exist.
\end{itemize}

While the second strategy is more general, the first is technically much easier to employ if such a vector $z_0$ exists. This is the case in all the examples below. 

\medskip

The following result appears to be new. 

\begin{theorem}[cf.~\mbox{\cite[Proposition~2.1]{M}}]
Let $p\equiv 1\ (\mathrm{mod}\ 4)$ be a prime number. There does not exist $BH(p+1,2p)$ matrices.
\end{theorem}
\begin{proof}
Assume by contradiction that a $BH(p+1, 2p)$ matrix exists, and consider the corresponding function $g$ in Theorem \ref{main}, and the $\cong$ equivalence classes of $\Omega_{2p}^{p+1}$.

Write $p+1=4L+2$. In this case, there are $K$ distinct $\cong$ equivalence classes of $ORT$, each with a representative of the form $\omega=[\omega_1,-\omega_1,\omega_2,-\omega_2,\dots,\omega_{2L+1},-\omega_{2L+1}]$, where $\omega_j\in \Omega_{2p}$. Therefore, with $z_0=[p,p,\dots, p]$ we have $\omega^{z_0}=(-1)^{p(2L+1)}\prod_{k=1}^{2L+1}\omega_k^{2p}=-1$. In particular, by equations  \eqref{g0hat2} and \eqref{lin2}  we have:
\[\widehat{g}(z_0)=1-\sum_{i=1}^Kc_i|\mathrm{Orb}(\omega_i)|=1-(p+1-1)=1-p<0,\]
a contradiction.
\end{proof}

We can also give a new proof of a special case of a result given in \cite[Theorem~7.9]{BBS}.
\begin{theorem}[Banica et al.~\cite{BBS}]
There exists no $BH(p+2,2p)$ matrix for any prime $p>2$.
\end{theorem}

\begin{proof}
It is easy to see that the set $ORT$ has a single $\cong$ equivalence class, represented by $\omega_1:=[1,\zeta_{2p}^2,\dots,\zeta_{2p}^{2p-2},1,-1]$. The size of $ORT$ is $\binom{p+2}{2}\cdot p!\cdot 2p$. Therefore, using \eqref{c12} we have 
\[c_1=\frac{p+1}{|ORT|}=\frac{p+1}{\binom{p+2}{2}\cdot p!\cdot 2p}.\]

Let $z_0:=[0,p,\dots,p]$, and substitute into \eqref{g0hat2} to obtain
\[\widehat{g}(z_0)=1+\frac{p+1}{\binom{p+2}{2}\cdot p!\cdot 2p}\sum_{x\cong \omega_1}x_1^p\prod_{k=1}^{p+2}x_k^{p},\]
where we have used the identity $x_1^0=x_1^p\cdot x_1^p$. 

Notice that the value of the appearing product is always $\pm 1$. Let us write $ORT=ORT^+\cup ORT^{-}$, and define $x\in ORT^+$ if $\prod_{k=1}^{p+2}x_k^p=1$, and $x\in ORT^-$ if this product is $-1$. Clearly $x\in ORT^+$ if and only if $\zeta_{2p}x\in ORT^{-}$. Consequently
\begin{equation}\label{ortpsize}
|ORT^+|=|ORT^-|=\binom{p+2}{2}\cdot p!\cdot p.
\end{equation}

We have:
\[\widehat{g}(z_0)=1+\frac{p+1}{\binom{p+2}{2}\cdot p!\cdot 2p}\left(\sum_{x\in ORT^+}x_1^p-\sum_{x\in ORT^-}x_1^p\right).\]
All elements of $ORT$ have the form $x=[a,a\zeta_{2p}^2,\dots,a\zeta_{2p}^{2p-2},b,-b]$, after a suitable permutation. The product of the entries of $x$ is $-a^{p}b^2\zeta_{2p}^{p(p-1)}=-a^pb^2$. Therefore the product of the $p$th powers is $-a^{p^2}=-a^{p(p-1)}a^p=-a^p$.

\medskip

Therefore, we have $\prod_{k=1}^{p+2}x_k^p=-1$ if and only if $a^{p}=1$, i.e., $a$ is an even $2p$th root of unity. This means that there is exactly one odd $2p$th root in $x$ (i.e.\ an entry $\mathrm{exp}(\pi\mathbf{i}t/p)$ where $t$ is odd). How often is this entry the very first coordinate? By easy combinatorial counting, the number of such vectors is $p!\cdot \binom{p+1}{2}$. Indeed, we have $p$ choices for the leading odd $2p$th root, say $\alpha$, which determines the value of the duplicate entry ($-\alpha$), two of which can be placed to any of the remaining $p+1$ places. The order of the remaining uniquely determined entries can be chosen in $(p-1)!$ ways. 

As such, within the set $ORT^-$ we have $x_1^p=-1$  exactly $p!\cdot \binom{p+1}{2}$ many times, and will be $+1$ in all the remaining cases. Thus, using \eqref{ortpsize}:
\[\sum_{x\in ORT^-}x_1^p=-p!\cdot \binom{p+1}{2}+\left(|ORT^-|-p!\cdot \binom{p+1}{2}\right)=\frac{(p+1)!\cdot p^2}{2}.\]
This determines the sum for the $ORT^+$ part, as well: 
\[\sum_{x\in ORT^+}x_1^p=-\sum_{x\in ORT^+}(\zeta_{2p} x_1)^p=-\sum_{y\in ORT^-} y_1^p=-\frac{(p+1)!\cdot p^2}{2}.\]
Altogether, we have:
\[\widehat{g}(z_0)=1-\frac{p+1}{\binom{p+2}{2}\cdot p!\cdot 2p}\cdot \frac{2(p+1)!\cdot p^2}{2}=1-\frac{p(p+1)}{p+2}=2-p-\frac{2}{p+2}<0,\]
a contradiction.
\end{proof}

Finally, we give a new proof of the nonexistence of $BH(5,12)$ matrices. The result itself is not new, because it follows from Haagerup's  complete classification of complex Hadamard matrices of order 5 in \cite{H}.   

\begin{theorem}[cf. ~\cite{H}]
\label{thm5}
There exist no $BH(5,12)$ matrices.
\end{theorem}

\begin{proof}
Let $n=5$, and $q=12$, and assume by contradiction that a $BH(5,12)$ matrix exists. It is easy to see that the number of $\cong$ equivalence classes of $ORT$ is $K=2$, represented by
$\omega_1:=[1,\zeta_{12}^4,\zeta_{12}^8,1,-1]$ and $\omega_2:=[1,\zeta_{12}^4,\zeta_{12}^8,\zeta_{12},-\zeta_{12}]$, with orbit sizes $720$ and $1440$, respectively. We have
\[g(x)=\begin{cases}
1/12,& \text{if } x\cong [1,1,1,1,1]\\
c_1,& \text{if } x\cong \omega_1\\
c_2,& \text{if } x\cong \omega_2\\
0,& \text{otherwise}.
\end{cases}\]
By the linear constraint \eqref{lin2}, we have $1+\sum_{i=1}^2 |\mathrm{Orb}(\omega_i)|c_i=5$, thus
\begin{equation}\label{lin4}
720c_1+1440c_2=4.
\end{equation}
Using \eqref{g0hat2} we have
\begin{equation}\label{512}
    \widehat{g}(z)=1+c_1\sum_{x\cong \omega_1}x^z+c_2\sum_{x\cong \omega_2}x^z.
\end{equation}
Choose $z_0=[1,1,1,10,11]\in\mathbb{Z}_{12}^5$. By direct evaluation (by computer) we have $\sum_{x\cong \omega_1}x^{z_0}=-216$, while $\sum_{x\cong \omega_2}x^{z_0}=-432.$ Therefore, regardless of the values of $c_1, c_2$, we conclude by 
\eqref{lin4} and \eqref{512} that 
\[\widehat{g}(z_0)=-1/5,\]
a contradiction.
\end{proof}

It is a natural question whether the converse of Theorem \ref{main} is true, i.e. whether the existence of a function $g$ with the prescribed properties implies the existence of a $BH(n,q)$ matrix. The following example shows, unfortunately, that this is not the case in general.
\begin{example}
For $(n,q)=(15, 3)$ a function $g$ with all the properties prescribed in Theorem \ref{main} exists, while a Butson Hadamard matrix $BH(15, 3)$ does not. 

\medskip

Indeed, there is a single $\cong$ equivalence class in $ORT$, represented by the element $\omega_1=[1,1,1,1,1,1,\zeta_3,\zeta_3,\zeta_3,\zeta_3,\zeta_3,\zeta_3^2,\zeta_3^2,\zeta_3^2,\zeta_3^2,\zeta_3^2]$. The size of $ORT$ is $|\mathrm{Orb}(\omega_1)|=\frac{15!}{5!\cdot 5!\cdot 5!}=756756$. We have $c_1=\frac{n-1}{|\mathrm{Orb}(\omega_1)|}=\frac{1}{54054}$, and
\[\widehat{g}(z)=1+\frac{1}{54054}\sum_{x\cong \omega_1}\prod_{k=1}^{15}x_k^{z_k}.\]

This can directly be evaluated at all $z\in \Z_3^{15}$ (keeping in mind that it is sufficient to evaluate one representative in each $\simeq$ equivalence class of $\Z_3^{15}$). The results show that $\widehat{g}(z)\ge 0$ for all choices of $z$. 

On the other hand, it is known that $BH(15, 3)$ matrices do not exist \cite[Example~2]{W}. 
\end{example}

The situation is completely similar in the case $(n,q)=(8,15)$: a function $g$ with all the properties prescribed in Theorem \ref{main} exists, while a Butson Hadamard matrix $BH(8,15)$ does not \cite[Theorem~4.11]{LOS}. 

\medskip

Finally, it is worth mentioning here that if for some values of $(n,q)$ there exists a function $g=g_{n,q}$ with all the properties prescribed in Theorem \ref{main}, then such a function also exists for any multiple of $q$, i.e.\ the method cannot prove the nonexistence of Butson Hadamard matrices $BH(n, kq)$ for any $k\ge 1$. For example, appropriate functions exist for the values $(15, 6), (15,9), (15, 12)$, or $(8,30), (8,45)$, etc. The reason is that we can consider the subgroup $\Omega_q^n\le \Omega_{kq}^n$ composed of vectors of length $n$ with each entry being a $q$th root of unity, and simply copy the function $g_{n,q}$ onto this subgroup. 

\medskip

It would be very interesting to understand what other properties of the function $g$ can guarantee the existence of a matrix $BH(n,q)$. 

\section{Construction of a matrix $BH(18, 14)$}\label{sec3}

In this section we use the necessary conditions given in Theorem \ref{main} to facilitate the search for a $BH(18, 14)$ matrix. The existence of such a matrix has been open so far (see \cite[Table~7]{LOS}).  

\medskip

Assume that a $BH(18,14)$ matrix $H$ exists, and consider the corresponding function $g_H$ defined 
in equation \eqref{gh}. Let $z_0=[7, 7, \dots, 7]$, and evaluate $\widehat{g}_H(z_0)$. By \eqref{ghhat} we have 
$$\widehat{g_H}(z_0)=\sum_{i=1}^n\sum_{j=1}^n(h_i\circ\overline{h}_j)^{z_0}
=\left|\sum_{i=1}^nh_i^{z_0}\right|^2\geq 0,$$
and all vectors $h_i\circ \overline{h}_j$ are elements of $ORT\cup\{e\}$. 

\medskip

For any vector $v\in \Z_{14}^{18}$ we have $v^{z_0}=\pm 1$, so the double sum consists of terms $\pm 1$ only. For terms $i=j$ we trivially have $(h_i\circ \overline{h}_j)^{z_0}=1$. Let $\zeta=\mathrm{exp}(2\pi\mathbf{i}/14)$. It is easy to see that the only elements $v$ of $ORT$ where $v^{z_0}=+1$ occurs are the vectors of the form
\[[1,-1,1,\zeta^2,\zeta^4,\zeta^6,\zeta^8,\zeta^{10},\zeta^{12},\zeta^y, -\zeta^{y},1,\zeta^2,\zeta^4,\zeta^6,\zeta^8,\zeta^{10},\zeta^{12}]\]
for some $0\le y\le 6$. Let us call the collection of these vectors $ORT^+$. The double sum being nonnegative, we conclude that $h_i\circ \overline{h}_j\in ORT^+$ or $i=j$ must account for at least half of the cases. This means that there must exist an $i$ such that $h_i\circ h_j\in ORT^+$ occurs for at least 8 different values of $j\ne i$. After permuting rows and multiplying the columns of the matrix $H$ by appropriate powers of $\zeta$, we can assume that $i=1$, and $h_1=[1,1,\dots, 1]$. 

\medskip

At this point it will be convenient to change notation slightly, and encode the rows of $H$ with the appearing exponents of $\zeta$. For instance, $h_1=[1,1, \dots, 1]$ will be encoded as $h_1'=[0,0,0,0,0,0,0,0,0,0,0,0,0,0,0,0,0,0]$, and the elements of $ORT^+$ will be encoded as \mbox{$[0,7,0,2,4,6,8,10,12,y,y+7,0,2,4,6,8,10,12]$}.
Based on examples of $BH(14, 10)$ matrices, we make a concrete choice, and  assume that the first two rows of $H$ are
$h_1'$ and $h_2'=[0,7,0,2,4,6,8,10,12,0,7,0,2,4,6,8,10,12]$, and the next seven rows are  permutations of $[0,7,0,2,4,6,8,10,12,y,y+7,0,2,4,6,8,10,12]$ for possibly different values of $y$. Also, looking at examples of $BH(14, 10)$ matrices, we made the further assumption that in rows $h_3', h_4', \dots, h_9'$ the values of $y$ are all distinct and range from $0$ to $6$ in some order, and the permutation of coordinates in each row permutes the left and right halves of the vector separately.

\medskip

As such, we assumed that the left half of each vector $h_3', \dots, h_9'$ is a permutation of \mbox{$[0,7,0,2,4,6,8,10,12]$}. Let us denote these half vectors by $h'_{3, L}, \dots, h'_{9,L}$. There was one more observation that facilitated the search. In the case of $BH(14, 10)$ matrices, the absolute value of the inner products of $|\langle h_{2,L}, h_{3,L}\rangle|$ and $|\langle h_{2,L}, h_{4,L}\rangle|$ was $2$ in some cases with $y=1$ and $y=\frac{q}{2}-1$ for $h'_3$ and $h'_4$, respectively. Therefore, in the search for $BH(18,14)$ we also assumed that $y=1$ for $h'_3$, and $y=6$ for $h'_4$, and that the corresponding inner products have absolute value $2$. Together with the orthogonality of the rows $h_1, h_2, h_3, h_4$ these assumptions restricted the search in a sufficient way, and only $134504$ triplets $(h_1,h_2,h_3)$ and $2421064$ quadruples of rows $(h_1, h_2, h_3, h_4)$ satisfied all of these constraints. 

\medskip

After that, adding further orthogonal rows $h_5, \dots, h_9$ with corresponding values ${y=0,2,3,4,5}$ the number of cases found was as follows: $401952$, $4304$, $6128$, $8976$, $8912$. As such, we ended up with $8912$ cases of orthogonal rows $(h_1, h_2, \dots, h_9)$. To complete the construction, we assumed the simplifying condition that $h'_{10}=[0,0,0,0,0,0,0,0,0,7,7,7,7,7,7,7,7,7]$. Indeed, by the algebraic structure $h_{10}$ is automatically orthogonal to the preceding rows. This implies that the left and right half of each of the remaining vectors $h'_{11},\dots,h'_{18}$ must be a permutation of $[a,a+7,b,b+2,b+4,b+6,b+8,b+10,b+12]$ for some $0\leq a\leq 6$, $0\leq b\leq 1$. In the end, each of the $8912$ cases could be completed to a Butson Hadamard matrix $BH(18,14)$ by adding further orthogonal rows $h_{11}, \dots, h_{18}$. Some of these matrices may be equivalent to each other, and we do not claim that they constitute a complete list of $BH(18,14)$ matrices. Nevertheless, studying these matrices may well lead to the resolution of Conjecture~\ref{conj1} in the future.

\medskip

Executing this search took about $4$ hours on a compute cluster with $192$ CPU cores.

\begin{example}
We present here one example of $BH(18, 14)$ matrices we discovered in logarithmic form, i.e., the matrix entries denote the exponents of $\zeta_{14}$. The horizontal and vertical lines partition the matrix highlighting four $7\times 7$ blocks.

\[H=\left[
\begin{array}{cc|ccccccc|cc|ccccccc}
 0 & 0 & 0 & 0 & 0 & 0 & 0 & 0 & 0 & 0 & 0 & 0 & 0 & 0 & 0 & 0 & 0 & 0 \\
 0 & 7 & 0 & 2 & 4 & 6 & 8 & 10 & 12 & 0 & 7 & 0 & 2 & 4 & 6 & 8 & 10 & 12 \\
\hline
 0 & 7 & 0 & 8 & 12 & 4 & 10 & 6 & 2 & 8 & 8 & 12 & 6 & 10 & 2 & 1 & 4 & 0 \\
 0 & 7 & 6 & 10 & 2 & 8 & 4 & 0 & 12 & 6 & 6 & 4 & 8 & 0 & 13 & 2 & 12 & 10 \\
 0 & 7 & 2 & 12 & 8 & 6 & 0 & 4 & 10 & 0 & 0 & 7 & 10 & 6 & 4 & 12 & 2 & 8 \\
 0 & 7 & 12 & 4 & 0 & 10 & 8 & 2 & 6 & 2 & 2 & 10 & 9 & 12 & 8 & 6 & 0 & 4 \\
 0 & 7 & 4 & 2 & 10 & 0 & 6 & 12 & 8 & 10 & 10 & 2 & 0 & 8 & 12 & 4 & 3 & 6 \\
 0 & 7 & 8 & 0 & 6 & 2 & 12 & 10 & 4 & 4 & 4 & 6 & 12 & 11 & 0 & 10 & 8 & 2 \\
 0 & 7 & 10 & 6 & 4 & 12 & 2 & 8 & 0 & 12 & 12 & 8 & 4 & 2 & 10 & 0 & 6 & 5 \\
 \hline
 0 & 0 & 0 & 0 & 0 & 0 & 0 & 0 & 0 & 7 & 7 & 7 & 7 & 7 & 7 & 7 & 7 & 7 \\
 0 & 7 & 7 & 9 & 11 & 13 & 1 & 3 & 5 & 7 & 0 & 0 & 2 & 4 & 6 & 8 & 10 & 12 \\
 \hline
 0 & 0 & 7 & 6 & 12 & 2 & 10 & 8 & 4 & 0 & 7 & 5 & 11 & 3 & 7 & 1 & 13 & 9 \\
 0 & 0 & 4 & 7 & 6 & 12 & 2 & 10 & 8 & 12 & 5 & 9 & 5 & 11 & 3 & 7 & 1 & 13 \\
 0 & 0 & 8 & 4 & 7 & 6 & 12 & 2 & 10 & 10 & 3 & 13 & 9 & 5 & 11 & 3 & 7 & 1 \\
 0 & 0 & 10 & 8 & 4 & 7 & 6 & 12 & 2 & 8 & 1 & 1 & 13 & 9 & 5 & 11 & 3 & 7 \\
 0 & 0 & 2 & 10 & 8 & 4 & 7 & 6 & 12 & 6 & 13 & 7 & 1 & 13 & 9 & 5 & 11 & 3 \\
 0 & 0 & 12 & 2 & 10 & 8 & 4 & 7 & 6 & 4 & 11 & 3 & 7 & 1 & 13 & 9 & 5 & 11 \\
 0 & 0 & 6 & 12 & 2 & 10 & 8 & 4 & 7 & 2 & 9 & 11 & 3 & 7 & 1 & 13 & 9 & 5 \\
\end{array}
\right].\]
For typographical reasons, we chose the (otherwise arbitrary) ordering of the rows $h'_{11},\dots, h'_{18}$ in a way so that circulant $7\times 7$ blocks to appear.
\end{example}

\medskip

\section{On $2$-circulant matrices of simple index $k$}\label{sec4}
In this section we recall the theory of circulant matrices of type `index $k$' as developed by Bj\"orck \cite{BJ}, and Haagerup \cite{bjhaa}, \cite{haagerupCyc}, and modify it accordingly to obtain several new examples of $2$-circulant Butson-type complex Hadamard matrices.

Let $p$ be a prime number, let $k$ be a positive integer dividing $p-1$, and let $\zeta_q:=\mathrm{exp}(2\pi\mathbf{i}/q)\in\mathbb{C}$ as before. Consider the unique subgroup $G_0$ of $(\mathbb{Z}_p^\ast;\cdot)$ of index $k$, namely the group of $k$th residues $G_0:=\{h^k\colon h\in \mathbb{Z}_p^\ast\}$. Let $g\in\mathbb{Z}_p^\ast$ be any generator, and let $G_{\ell}:=g^{\ell}\cdot G_0$, $\ell\in\{1,\dots,k-1\}$ be the $k-1$ other cosets of $G_0$ in $\mathbb{Z}_p^\ast$. Let $X$ be a circulant matrix of order $p$ whose first row $x=[x_0,x_1,\dots,x_{p-1}]$ is constant on the cosets of $G_0$, namely, $x_0:=1$, and $x_i=\zeta_q^{c_{\ell}}$ if $i\in G_\ell$ for every $i\in\{1,\dots, p-1\}$, where $[c_0,\dots,c_{k-1}]\in\mathbb{Z}_q^k$. Bj\"orck and Haagerup \cite{bjhaa} called the vectors $x$ `simple of index $k$', and systematically studied the cases $k\leq 3$ leading to circulant complex Hadamard matrices. The case $k=4$ was studied to some extent in \cite{phd}. Furthermore, circulant Butson-type matrices were studied in \cite{arasu} and \cite{slenker}.

\medskip

Here we consider the related problem of constructing two circulant matrices $X$ and $Y$ whose first rows are simple of index $k$, satisfying
\begin{equation}\label{star}
XX^\ast + YY^\ast=2pI_p.
\end{equation}
Since circulant matrices commute, for $q$ even this equation leads to a $2$-circulant $BH(2p,q)$ matrix of the form
\[H=\left[\begin{array}{cc} X & Y\\
Y^\ast & -X^\ast\end{array}\right].\]
Further, multiplying \eqref{star} by the all $1$ matrix $J$ from the left and right, we obtain a necessary condition relating the elements of $X$ and $Y$:
\begin{equation}\label{starcond}
\left|\sum_{i=0}^{p-1}x_i\right|^2+\left|\sum_{i=0}^{p-1}y_i\right|^2=2p.
\end{equation}
We attempt to construct such a matrix $H$ by symbolically computing the \textit{periodic autocorrelation coefficients} of the matrix $X$ for $s\in\{1,\dots,p-1\}$:
\[\gamma_s([c_0,c_1,\dots,c_{k-1}]):=\sum_{i=0}^{p-1}x_i\overline{x}_{i+s}\]
(where indices are taken modulo $p$), and put these in a vector $\Gamma(c):=[\gamma_1,\dots,\gamma_{p-1}]$. Our goal is to find \textit{complementary pairs} $a,b\in\mathbb{Z}_q^k$ such that $\Gamma(a)+\Gamma(b)=[0,\dots,0]$. In order to do this, we first create an empty set $S$, and then for every $c\in \mathbb{Z}_q^k$ we update $S=S\cup\{\Gamma(c)\}$ and then test whether $-\Gamma(c)\in S$. As soon as this latter holds, we discover a choice $a:=c$ having a complementary pair $b$ tested earlier. Creating an efficient data structure (i.e., a hash table) storing $\Gamma(c)$ with the corresponding $c$ can be conveniently facilitated by rewriting \eqref{starcond} in terms of $a$ and $b$. In particular, for a given $a$ a complementary $b$ must satisfy
\[\left|k+(p-1)\sum_{i=0}^{k-1}b_i\right|^2=2k^2p-\left|k+(p-1)\sum_{i=0}^{k-1}a_i\right|^2.\]
We remark that due to the imposed algebraic structure many of the autocorrelation coefficients coincide. We omit the details, and refer the interested reader to \cite[Section~7]{haagerupCyc}. As a toy example, we included in Table~\ref{T1} below the parameters of a newly discovered $BH(34,10)$ matrix.

\medskip

It is well-known that for prime numbers $p>2$ a $BH(2p,2)$ matrix cannot exist. Furthermore, $BH(2p,4)$ matrices have been extensively studied earlier \cite{holzmann}, \cite{DJ}, \cite{sebyam}, with the smallest outstanding case now being $BH(94,4)$. Since $47-1=2\cdot 23$, the method outlined here is not suitable to address this case. Therefore, our focus is on the more approachable $BH(2p,6)$ matrices. For $p=17$ and $p=29$ index $4$ type $BH(2p,6)$ matrices were constructed in \cite[p.~31]{phd}. Here we extend this list with three new examples.
\begin{theorem}
There exist $2$-circulant $BH(62,6)$, $BH(82,6)$, and $BH(146,6)$ matrices.
\end{theorem}
\begin{proof}
See Table~\ref{T1} for a list of complementary pairs $a$ and $b$, respectively.
\end{proof}

\begin{table}[htbp]
\begin{tabular}{ccccccc}
\toprule
$(n,q)$ & $p$ & $k$ & $g$ & $G_0$ & $a$ & $b$\\
\midrule
$(34,10)$ & $17$ & $4$ & $3$ & $\{1,4,13,16\}$ & $[8,2,6,4]$ & $[9,3,7,1]$\\
$(62,6)$ & $31$ & $6$ & $3$ & $\{1,2,4,8,16\}$ & $[0,2,2,4,5,1]$ & $[0,2,5,2,5,2]$\\
$(82,6)$ & $41$ & $8$ & $6$ & $\{1,10,16,18,37\}$ & $[5,4,1,2,5,4,1,2]$ & $[3,3,0,3,3,0,0,3]$\\
$(146,6)$ & $73$ & $8$ & $5$ & $\{1,2,4,8,16,32,37,55,64\}$ & $[3,5,1,3,5,3,5,1]$ & $[5,3,5,1,3,5,1,3]$\\
\bottomrule
\end{tabular}
\caption{Parameters of $2$-circulant $BH(n,q)$ matrices.}\label{T1}
\end{table}

Comparing the tables on \cite[p.~32]{phd} and \cite[p.~111]{guillermo}, it seems that $BH(46,6)$, $BH(74,6)$, $BH(86,6)$, and $BH(94,6)$ are the remaining undecided even orders less than $100$. The smallest outstanding odd order remains the challenging prime order case $BH(31,6)$.

We are certain that this method will eventually lead to further unexpected discoveries. Therefore, a systematic computer search for all $2$-circulant $BH(2p,q)$ matrices with reasonably small $p$, $q$ and $k$ is timely.

\section{Acknowledgement}

The authors are grateful to D\'aniel Varga for helpful discussions.

\medskip

\noindent
{\sc Domonkos Czifra}

\noindent
{\em HUN-REN Alfréd Rényi Institute of Mathematics, Reáltanoda u. 13-15, 1053,  Budapest, Hungary}

\noindent
e-mail address: \texttt{czifra.domonkos@renyi.hu}

\medskip

\noindent
{\sc Máté Matolcsi} 

\noindent
{\em HUN-REN Alfréd Rényi Institute of Mathematics, Reáltanoda u. 13-15, 1053, Budapest, Hungary, and\\
Department of Analysis and Operations Research,
Institute of Mathematics,
Budapest University of Technology and Economics,
Műegyetem rkp. 3., H-1111 Budapest, Hungary}

\noindent
e-mail address: \texttt{matolcsi.mate@renyi.hu}

\medskip

\noindent
{\sc Ferenc Sz\"oll\H osi}

\noindent
{\em Interdisciplinary Department of Science and Engineering, Shimane University,\\
1060 Nishikawatsu-cho, Matsue, Shimane, 690-8504, Japan}

\noindent
e-mail address: \texttt{szollosi@riko.shimane-u.ac.jp}
\end{document}